\documentclass[11pt]{amsart}
\usepackage{latexsym,amsfonts,amssymb}
\usepackage{hyperref}

\newtheorem{theorem}{Theorem}[section]
\newtheorem{lemma}[theorem]{Lemma}
\newtheorem{corollary}[theorem]{Corollary}
\newtheorem{question}[theorem]{Question}
\theoremstyle{definition}
\newtheorem{definition}[theorem]{Definition}

\theoremstyle{remark}

\newtheorem{example}[theorem]{Example}

\begin{document}
\noindent  \vspace{0.5in}

\title[Open uniform (G) at non-isolated points and maps]%
{Open uniform (G) at non-isolated points and maps}

\author{Fucai Lin}
\address{(Fucai Lin)Department of Mathematics and Information Science,
Zhangzhou Normal University, Zhangzhou 363000, P. R. China}
\email{linfucai2008@yahoo.com.cn}

\author{Shou Lin}
\address{(Shou Lin)Institute of Mathematics, Ningde Teachers' College, Ningde, Fujian
352100, P. R. China; Department of Mathematics and Information
Science, Zhangzhou Normal University, Zhangzhou 363000, P. R. China}
\email{linshou@public.ndptt.fj.cn}
\thanks{Supported by the NSFC(No. 10971185) and the Educational Department of Fujian Province(No. JA09166).}


\begin{abstract} In this paper, we mainly introduce the notion of
an open uniform (G) at non-isolated points, and show that a space
$X$ has an open uniform (G) at non-isolated points if and only if
$X$ is the open boundary-compact image of metric spaces. Moreover,
we also discuss the inverse image of spaces with an open uniform (G)
at non-isolated points. Two questions about open uniform (G) at
non-isolated points are posed.
\end{abstract}

\maketitle

\parskip 0.15cm

\section{Introduction}
In \cite{LL}, F. C. Lin and S. Lin defined the notion of uniform
bases at non-isolated points, and obtained that a space $X$ has an
uniform base at non-isolated points if and only if $X$ is the open
and boundary-compact image of metric spaces. Isbell-Mr\'{o}wka space
$\psi(D)$\cite{MS} has an uniform base at non-isolated points, and
however, it has not any uniform base. It is well known that a space
has an uniform base if and only if it has an open uniform (G) if and
only if it is the open compact image of metric spaces. Therefore, we
generalize the notion of open uniform (G), and define the notion of
the open uniform (G) at non-isolated points such that a space has an
open uniform (G) at non-isolated points if and only if it has an
uniform base at non-isolated points. In \cite{LL1}, F. C. Lin and S.
Lin have discussed the image of spaces with an uniform base at
non-isolated points. In this paper, we also also discuss the inverse
image of spaces with an uniform base at non-isolated points.

By $\mathbb{R, N}$, denote the set of all real numbers and positive
integers, respectively. For a topological space $X$, let $\tau (X)$
denote the topology for $X$, and let
$$I(X)=\{x:x \mbox{ is an isolated point of } X\},$$
$$X^{d}=X-I(X),$$
$$\mathcal{I}(X)=\{\{x\}:x\in I(X)\}.$$

In this paper all spaces are Hausdorff, all maps are continuous and
onto. Recall some basic definitions.

\begin{definition}\label{d0}
Let $\mathcal{P}$ be a base of a space $X$. $\mathcal{P}$ is an {\it
uniform base} \cite{Al}~(resp. {\it uniform base at non-isolated
points} \cite{LL}) for $X$ if for each ($resp.$\ non-isolated) point
$x\in X$ and any countably infinite subset $\mathcal{P}^{\prime}$ of
$\{P\in\mathcal{P}:x\in P\}$, $\mathcal{P}^{\prime}$ is a
neighborhood base at $x$ in $X$.
\end{definition}

\begin{definition}\label{d1}
A space $X$ has an {\it open uniform (G)}\cite{MR} (resp. {\it open
uniform (G) at non-isolated points}), if there exists a collection
$\mathcal{W}=\{\mathcal{W}_{x}: x\in X\}$ of open subsets of $X$
satisfying the following conditions:
\begin{enumerate}
\item For each $x\in X$, $x\in \cap\mathcal{W}_{x}$ and $|\mathcal{W}_{x}|\leq
\aleph_{0}$;

\item For each $x\in U\in\tau (X)$, there exists an open neighborhood
$V(x, U)$ of $x$ such that there is a $W\in\mathcal{W}_{y}$ with
$x\in W\subset U$ for each $y\in V(x, U)$ ($y\in V(x, U)\cap
X^{d}$);

\item For each $x\in X$, $\mathcal{W}_{x}^{\prime}$ is a network at
point $x$ for any infinite subfamily
$\mathcal{W}_{x}^{\prime}\subset \mathcal{W}_{x}$.
\end{enumerate}
\end{definition}

In the Definitions~\ref{d0} and~\ref{d1}, ``at non-isolated points''
means ``at each non-isolated point of $X$''. If
$\mathcal{W}=\{\mathcal{W}_{x}: x\in X\}$ is an open uniform (G) at
non-isolated points, then $(\mathcal{W}\setminus\{\mathcal{W}_{x}:
x\in I\})\cup\{\mathcal{W}_{x}^{\prime}=\{x\}: x\in I\}$ is also an
open uniform (G) at non-isolated points for $X$. Therefore, we
always suppose that $\mathcal{W}_{x}=\{x\}$ if $x\in I$ in this
paper. It is obvious that spaces with an open uniform (G) have an
open uniform (G) at non-isolated points.

\begin{definition} Let $f:X\rightarrow
Y$ be a map.
\begin{enumerate}
\item $f$ is a {\it compact map} if each $f^{-1}(y)$ is compact in $X$;

\item $f$ is a {\it boundary-compact map}, if each
$\partial f^{-1}(y)$ is compact in $X$;

\item $f$ is a {\it perfect map} if it is a closed and compact map.

\item $f$ is called a {\it $\leq k$-to-one} (resp. {\it $k$-to-one},
{\it finite-to-one) map} if $|f^{-1}(y)|\leq k$ (resp. $|f^{-1}(y)|=
k$, $f^{-1}(y)$ is finite) for every $y\in Y$, where $k\in
\mathbb{N}$;

\item $f$ is called a {\it local homeomorphism} if, for each $x\in
X$, there exists an open neighborhood $U$ of $x$ in $X$ such that
$f|U:U\rightarrow f(U)$ is a homeomorphism map and $f(U)$ is open in
$Y$.

\item $f$ is an {\it irreducible map} if there does not exist a
proper closed subset $X'$ of $X$ such that $f(X')=Y$.
\end{enumerate}
\end{definition}

\begin{definition}\cite{LL}
Let $X$ be a space and $\{\mathcal{P}_{n}\}_{n}$ a sequence of
collections of open subsets of $X$. $\{\mathcal{P}_{n}\}_{n}$ is
called a {\it development at non-isolated points} for $X$ if
$\{\mbox{st}(x,\mathcal{P}_{n})\}_{n}$ is a neighborhood base at $x$
in $X$ for each non-isolated point $x\in X$. $X$ is called {\it
developable at non-isolated points} if $X$ has a development at
non-isolated points.
\end{definition}

\begin{definition}\cite{LL}
Let $\mathcal P$ be a family of subsets of a space $X$.\
$\mathcal{P}$ is called {\it point-finite at non-isolated
points}(resp. {\it point-countable at non-isolated points}) if for
each non-isolated point $x\in X$, $x$ belongs to at most finitely
(countably) many elements of $\mathcal{P}$.\ Let $\{\mathcal
P_n\}_{n}$ be a development at non-isolated points for $X$.
$\{\mathcal P_n\}_{n}$ is said to be {\it a point-finite development
at non-isolated points} for $X$ if each $\mathcal P_n$ is
point-finite at each non-isolated point of $X$.
\end{definition}

\begin{definition}
Let $X$ be a topological space. $g:\mathbb{N}\times X\rightarrow
\tau(X)$ is called a $g$-function, if $x\in g(n, x)$ and $g(n+1,
x)\subset g(n, x)$ for any $x\in X$ and $n\in \mathbb{N}$. For
$A\subset X$, put $$g(n,A)=\bigcup_{x\in A}g(n, x).$$
\end{definition}

Readers may refer to \cite{En, Ls} for unstated definitions and
terminology.

\vskip 0.5cm
\section{Open uniform (G) at non-isolated points}
In this section, we mainly show that a space has an open uniform (G)
at non-isolated points if and only if it has an uniform base at
non-isolated points. Firstly, we give some technique lemmas.

\begin{lemma}\cite{LL1}\label{l4}
Let $X$ be a topological space. Then the following conditions are
equivalent:
\begin{enumerate}
\item $X$ is an open boundary-compact image of a metric space;

\item $X$ has an uniform base at non-isolated points;

\item $X$ has a point-finite development at non-isolated points;

\item $X$ has a development
at non-isolated points, and $X^{d}$ is a metacompact subspace of
$X$.
\end{enumerate}
\end{lemma}

\begin{lemma}\label{l0}
Let $X$ have an open uniform (G) at non-isolated points. Then there
exists a $g$-function such that for each $x\in X^{d}$ and any
sequence $\{x_{n}\}_{n}$ with $x_{n}\in g(n, x)$ or $x\in g(n,
x_{n})$, $\{x_{n}\}_{n}$ has a subsequence converging to $x$.
\end{lemma}

\begin{proof}
Let $\mathcal{W}=\{\mathcal{W}_{x}: x\in X\}$ be an open uniform (G)
at non-isolated points for $X$.

Claim 1: There exists a sequences $\{\mathcal{H}_{n}\}_{n}$ of open
coverings of $X$, where $\mathcal{H}_{n}$ is point-finite at
non-isolated points for each $n\in \mathbb{N}$.

For each $x\in X$, let $\{G(n, x)\}_{n}$ be a decreasing open
neighborhood base at $x$, where, for each $x\in \mathbb{N}$, $G(n,
x)=\{x\}$ if $x\in I$. Next, we define the point-finite open
covering $\mathcal{H}_{n}$ at non-isolated points, $h_{n}:
\mathcal{H}_{n}\rightarrow X$ and open neighborhood $O(n ,x)$ of $x$
for each $x\in X$ by induction on $n\in \mathbb{N}$. Firstly, let
$\mathcal{H}_{0}=\{X\}$, and choose a point $z\in X$ and define
$h_{0}: \mathcal{H}_{0}\rightarrow X$ with $h_{0}(X)=z$. Put
\[O(1, x)=\left\{
\begin{array}{lll}
G(1, x), & x=z,\\
G(1, x)-\{z\}, & x\neq z.\end{array}\right.\] Suppose that we have
defined $\mathcal{H}_{m-1}, h_{m-1}$, and $O(m, x)$ for each $m\leq
n$ and $x\in X$. We endow $\mathcal{H}_{m-1}$ with a well-order by
$(\mathcal{H}_{m-1}, <)$. For each $H\in \mathcal{H}_{n-1}$, since
$X^{d}$ is hereditarily metacompact, there exists an open covering
$\mathcal{F}_{n}(H)$ of $H$ such that $\mathcal{F}_{n}(H)$ is
point-finite at non-isolated points and refines $\{H\cap V(x, O(n,
x))\}_{x\in H}$, where $V(x, O(n, x))$ is the open neighborhood of
$x$ stated in (3) of Definition 1.2. Put

$\mathcal{H}_{n}(H)=\mathcal{F}_{n}(H)-\cup\{\mathcal{F}_{n}(H^{\prime}):
H^{\prime}< H\}, H\in\mathcal{H}_{n-1};$

$\mathcal{H}_{n}=\cup\{\mathcal{H}_{n}(H):
H\in\mathcal{H}_{n-1}\}.$\\
Then $\mathcal{H}_{n}$ is an open covering of $X$, which is
point-finite at non-isolated points. For each $H\in\mathcal{H}_{n}$,
there exists just one $H^{\prime}\in\mathcal{H}_{n-1}$ such that
$H\in \mathcal{H}_{n}(H^{\prime})\subset
\mathcal{F}_{n}(H^{\prime})$. Then we can choose a point $x_{H}\in
H^{\prime}$ such that $H\subset H^{\prime}\cap V(x_{H}, O(n,
x_{H}))\subset O(n, x_{H})$. Define

$h_{n}(H)=x_{H};$

$O(n+1, x)=G(n+1, x)-\{h_{m}(H): m\leq n, H\in
(\mathcal{H}_{m})_{x}\ \mbox{and}\ x\neq h_{m}(H)\}.$\\
If $x\in X^{d}$, then $x\in O(n+1, x)\in\tau (X)$; if $x\in I$, then
$G(n+1, x)=O(n+1, x)=\{x\}\in\tau (X)$.

Claim 2: For each $x\in X^{d}$,
$X^{d}\cap\bigcap_{n=0}^{\infty}\mbox{st}(x,
\mathcal{H}_{n})=\{x\}.$

Suppose not, there exist distinct points $x, y\in X^{d}$ and a
sequence $\{H_{n}\}_{n}$ of subsets of $X$ such that $x, y\in
H_{n}\in\mathcal{H}_{n}$ for each $n\in \mathbb{N}$. For each
$n\in\mathbb{N}$, there exists just one sequence
$\{H_{n}^{m}\}_{m\leq n}$ such that $H_{n}^{n}=H_{n}, H_{n}^{m}\in
\mathcal{H}_{m}$ and $H_{n}^{m}\in\mathcal{H}_{m}(H_{n}^{m-1})$ for
each $1<m\leq n$. Then $x\in H_{n}\subset H_{n}^{m}\subset
H_{n}^{m-1}$. Since $\mathcal{H}_{m}$ is point-finite at point $x$,
we can define $I_{n}\subset \mathbb{N}$ and $i_{n}\in\mathbb{N}$ by
induction on $n\in\mathbb{N}$ as follows.

(1) $i_{n}=\mbox{min}I_{n};$

(2) $I_{n+1}\subset I_{n}-\{i_{n}\};$

(3) $m, k\in I_{n}\Rightarrow H_{m}^{n}=H_{k}^{n}.$\\
Let $K_{n}=H_{i_{n}}^{n}, q_{n}=h_{n}(K_{n})$. Then
$K_{n}=H_{m}^{n}$ for each  $m\in I_{n}$ and $q_{n}\in
K_{l}\supset K_{n}$ for each $l< n$. Since $x\in K_{n}\subset
V(q_{n}, O(n, q_{n}))$, there exists a $W_{n}\in\mathcal{W}_{x}$
such that $q_{n}\in W_{n}\subset O(n, q_{n})$ by the definition of
open uniform $(G)$.

Choose disjoint open sets $U_{x}$ and $U_{y}$ such that $x\in U_{x}$
and $y\in U_{y}$. Without loss of generality, we can assume that
there exists an infinite $J\subset \mathbb{N}$ such that, for each
$n\in J$, $q_{n}\not\in U_{x}$. Then $W_{n}\nsubseteq U_{x}$, and
therefore, $\{W_{n}: n\in J\}$ is finite. Hence, without loss of
generality, we can suppose that $W_{n}=W_{m}$ for any $n, m\in J$.
Thus $q_{m}\in O(i_{n}, q_{n})$, and $q_{m}=q_{n}$ by the definition
of $O(i_{n}, q_{n})$. Let $q_{n}=q$ for each $n\in J$. For each
$n\in J$, $x\in V(q, O(n, q))\subset O(n, q)\subset G(n, q)$, and
therefore, $x\in\bigcap_{n\in J}G(n, q)=\{q\}$, which is a
contradiction.

Now, we begin to show the Lemma. For each $n\in\mathbb{N}$, $x\in
X^{d}$, choose an $H(n, x)\in (\mathcal{H}_{n})_{x}$. For each $x\in
X$, define $g(n, x)$ by induction on $n$ as follows.
\[g(n+1, x)=\left\{
\begin{array}{lll}
V(x, H(n+1, x))\cap G(n+1, x)\cap g(n, x), & x\in X^{d},\\
\{x\}, & x\in I,\end{array}\right.\] where
\[g(1, x)=\left\{
\begin{array}{lll}
V(x, H(1, x))\cap G(1, x), & x\in X^{d},\\
\{x\}, & x\in I.\end{array}\right.\]
 Let $x\in X^{d}$ and
$\{x_{n}\}_{n}$ be a sequence with $x_{n}\in g(n, x)$ or $x\in g(n,
x_{n})$. We consider the following two cases.

Case 1: $\{n: x_{n}\in g(n, x)\}$ is infinite.

In this case, it is easy to show that the subsequence $\{x_{n}:
x_{n}\in g(n, x)\}$ converges to $x$.

 Case 2: $\{n: x_{n}\in g(n,
x)\}$ is finite.

In this case, we may assume that $x\in g(n, x_{n})$ for each $n\in
\mathbb{N}$. We show that the sequence $\{x_{n}\}_{n}$ itself
converges to $x$. Otherwise, there exists an open neighborhood $U$
of $x$ such that $\{x_{n}\}_{n}$ is not eventually in $U$. For each
$n\in \mathbb{N}$, since $x\in g(n, x_{n})$, we have $x_{n}\in
X^{d}$ and $x\in V(x_{n}, H(n, x_{n}))$. Hence, for each $n\in
\mathbb{N}$, there is a $W_{n}\in \mathcal{W}_{x}$ such that
$x_{n}\in W_{n}\subset H(n, x_{n})\subset \mbox{st}(x,
\mathcal{H}_{n})$. Let $M=\{n\in \mathbb{N}: x_{n}\not\in U\}$. Then
$M$ is infinite. Therefore, by the condition (3)in Definition 1.2,
$\{W_{n}: n\in M\}$ is finite set. Without loss of generality, we
can assume that $W_{n}=W_{m}$ for $n, m\in M$. Then,
$x_{n}\in\mbox{st}(x, \mathcal{H}_{m})$ for any $n, m\in
M$. Hence, $x_{n}\in\bigcap_{m\in M}\mbox{st}(x,
\mathcal{H}_{m})\cap X^{d}=\{x\}$ by Claim 2, which is a
contradiction.
\end{proof}

\begin{lemma}\label{l1}
If $X$ has an open uniform (G) at non-isolated points, then $X$ has
a point-countable base at non-isolated points.
\end{lemma}

\begin{proof}
Let $\mathcal{W}=\{\mathcal{W}_{x}: x\in X\}$ be an open uniform (G)
at non-isolated points for $X$, where $\mathcal{W}_{x}=\{W(n,
x)\}_{n}$. Let $g$ be a $g$-function satisfying the conditions in
Lemma~\ref{l0}. For each $n\in \mathbb{N}$ and the open covering
$\{g(n ,x): x\in X\}$ of $X$, since $X^{d}$ is metacompact, there
exists an open covering $\mathcal{U}_{n}$ such that
$\mathcal{U}_{n}$ is point-finite at non-isolated points and refines
$\{g(n ,x): x\in X\}$. For each $U\in \mathcal{U}_{n}$, there is a
$x_{U}\in X$ such that $U\subset g(n, x_{U})$. Let

$\mathcal{B}_{n, m}=\{U\cap W(m, x_{U}): U\in \mathcal{U}_{n}\},
m\in\mathbb{N}$;

$\mathcal{B}=\bigcup_{n, m\in \mathbb{N}}\mathcal{B}_{n, m}$.\\
Then $\mathcal{B}$ is an open collection of subsets of $X$ and
point-countable at non-isolated points of $X$. We now show that
$\mathcal{B}\cup \mathcal{I}(X)$ is point-countable base at
non-isolated points for $X$. Indeed, for each $x\in X^{d}$ and $x\in
O\in\tau (X)$, choose an $U_{n}\in (\mathcal{U}_{n})_{x}$ for each
$n\in \mathbb{N}$. We denote $x_{n}=x_{U_{n}}$. Then $x\in g(n,
x_{n})$, and hence sequence $\{x_{n}\}_{n}$ converges to $x$.
Therefore, there exists an $i\in \mathbb{N}$ such that $x_{i}\in
V(x, O)$. Since $x\in g(i, x_{i})$, we have $x_{i}\in X^{d}$ and
there is an $m\in \mathbb{N}$ such that $x\in W(m, x_{i})$. Thus
$x\in U_{i}\cap W(m, x_{i})\subset O$.
\end{proof}

Put $R^{+}=\{x\in\mathbb{R}: x\geq 0\}$.

\begin{lemma}\label{l2}
If $X$ has an open uniform (G) at non-isolated points, then there
exists a function $d: X\times X\rightarrow R^{+}$ such that, for
each $x\in X^{d}$, $x\in B(x, \frac{1}{n})$ and $\{\mbox{int}(B(x,
\frac{1}{n}))\}_{n}$ is a decreasing neighborhood base at $x$, where
$B(x, \frac{1}{n})=\{y\in X:d(x, y)<\frac{1}{n}\}$.
\end{lemma}

\begin{proof}
Let $g$ be the $g$-function constructed in the proof of
Lemma~\ref{l0}. For any distinct points $x, y\in X$, put
$$m(x, y)=\mbox{min}\{n\in \mathbb{N}: y\not\in g(n, x)\ \mbox{and}\ x\not\in g(n, y)\}.$$
Define $d: X\times X\rightarrow R^{+}$ as follows.
\[d(x, y)=\left\{
\begin{array}{lll}
0, & x=y,\\
\frac{1}{m(x, y)}, & x\neq y.\end{array}\right.\] Then, for each
point $x\in X^{d}$ and $n\in \mathbb{N}$, $x\in \mbox{int}(B(x,
\frac{1}{n}))$. Indeed, since $m(x, y)> n$ for each $y\in g(n, x)$,
$d(x, y)<\frac{1}{n}$. Then $y\in B(x, \frac{1}{n})$, and therefore,
$x\in g(n, x)\subset B(x, \frac{1}{n})$. It follows that
$x\in\mbox{int}(B(x, \frac{1}{n}))$. For each $x\in X^{d}$ and $x\in
U$, there exists an $m\in \mathbb{N}$ such that $B(x,
\frac{1}{m})\subset U$. Otherwise, suppose that $B(x,
\frac{1}{m})\nsubseteq U$ for each $m\in \mathbb{N}$. Choose a point
$x_{m}\in B(x, \frac{1}{m})\setminus U$ for each $m\in \mathbb{N}$.
Then $d(x, x_{m})<\frac{1}{m}$, and hence $x\in g(m, x_{m})$ or
$x_{m}\in g(m, x)$. By Lemma~\ref{l0}, $\{x_{m}\}_{m}$ has a
subsequence converging to $x$. It contradicts the fact that
$x_{m}\not\in U$ for each $m\in \mathbb{N}$.
\end{proof}

\begin{lemma}\label{l3}
If $X$ has an open uniform (G) at non-isolated points, then $X$ is a
developable space at non-isolated points.
\end{lemma}

\begin{proof}
By Lemma~\ref{l1}, let $\mathcal{U}$ be a point-countable base at
non-isolated points for $X$. Endow $X^{d}$ with a well-order by
$(X^{d}, <)$. Let $d: X\times X\rightarrow R^{+}$ be the function
defined in the proof of Lemma~\ref{l2}. For each $x\in X^{d}$, let
$(\mathcal{U})_{x}=\{U_{n}(x)\}_{n}$. For each $n\in \mathbb{N}$,
put

$V_{n}(x)=\mbox{int}(B(x, \frac{1}{n}))$;

$h(n, x)=U_{n}(x)\cap V_{n}(x)$;

$p(n, x)=\mbox{min}\{y\in X^{d}: x\in h(n, y)\};$

$g(n, x)=V_{n}(x)\cap (\cap\{h(i, p(i, x)): i\leq n\})\cap
(\cap\{U_{j}(p(i, x)): i, j\leq n, x\in U_{j}(p(i, x))\})$;

$\varphi_{n}=\{g(n, x): x\in X^{d}\}\cup \{g(n, x)=\{x\}: x\in
I\}$.\\
Then $\{\varphi_{n}\}_{n}$ is a development at non-isolated points.
Indeed, suppose not, there exists a point $x\in X^{d}$ and an open
neighborhood $U$ of $x$ such that there is $x_{i}\in X^{d}$
satisfying $x\in g(i, x_{i})\nsubseteq U$ for each $i\in
\mathbb{N}$. Since $x\in V_{i}(x_{i})$, $x_{i}\rightarrow x$. It
follows from Lemma~\ref{l2} that there exist $l, m\in \mathbb{N}$
such that $B(x, \frac{1}{l})\subset U_{m}(x)\subset U$. For each
$y\in X^{d}$, if $x\in h(l, y)\subset V_{l}(y)$, then $y\in B(x,
\frac{1}{l})\subset U_{m}(x)$. It follows that $p(l, x)\in
U_{m}(x)$, and therefore, there exists a $k\in \mathbb{N}$ such that
$U_{m}(x)=U_{k}(p(l, x))$. Since $U_{k}(p(l, x))\cap h(l, p(l, x))$
is an open neighborhood at $x$, there is an $i_{0}\in \mathbb{N}$
such that, for each $i\geq i_{0}$, $x_{i}\in U_{k}(p(l, x))\cap h(l,
p(l, x))$. Thus $p(l, x_{i})\leq p(l, x)$ for $i\geq i_{0}$, and on
the other hand, $x\in g(i, x_{i})\subset h(l, p(l, x_{i}))$ for
$i\geq l$. Then $p(l, x)\leq p(l, x_{i})$ for each $i\geq l$.
Therefore, $p(l, x_{i})=p(l, x)$ for $i\geq \mbox{max}\{i_{0}, l\}$.
It follows that $x_{i}\in U_{k}(p(l, x_{i}))$, and therefore, for
$i\geq \mbox{max}\{i_{0}, l, k\}$, $g(i, x_{i})\subset U_{k}(p(l,
x_{i}))=U_{k}(p(l, x))=U_{m}(x)\subset U$, which is a contradiction.
\end{proof}

\begin{theorem}
A space $X$ has an open uniform (G) at non-isolated points if and
only if
 $X$ has an uniform base at non-isolated points.
\end{theorem}

\begin{proof}
Necessity. By Lemma~\ref{l3}, $X$ has a development at non-isolated
points. Since $X^{d}$ is metacompact, $X$ has an uniform base at
non-isolated points by Lemma~\ref{l4}.

Sufficiency. Let $\mathcal{B}$ be an uniform base at non-isolated
points for $X$. If $\mathcal{B}$ is point-countable at non-isolated
points for $X$, then $\mathcal{W}=\{(\mathcal{B})_{x}: x\in
X^{d}\}\cup\mathcal{I}(X)$ is an open uniform (G) at non-isolated
points for $X$. Suppose that there exists a point $x\in X^{d}$ such
that $(\mathcal{B})_{x}$ is uncountable. If $z\in X-\{x\}$, then
$\{B\in (\mathcal{B})_{x}: z\in B\}$ is finite. Hence there are an
infinite subset $\{B_{n}: n\in \mathbb{N}\}\subset
(\mathcal{B})_{x}, x_{n}\in B_{n}-\{x\}$ for each $n\in \mathbb{N}$,
and some $k\in \mathbb{N}$ such that $x_{n}$ belongs to just $k$
many elements of $(\mathcal{B})_{x}$. Then $x_{n}\rightarrow x$ as
$n\rightarrow\infty$. Since $\mathcal{B}$ is a base for $X$, there
exists an infinite subfamily $\{B_{i}^{\prime}: i\in \mathbb{N}\}$
of $\mathcal{B}$ and a subsequence $\{x_{n_{i}}\}_{i}$ such that
$\{x_{n_{j}}: j\geq i\}\subset B_{i}^{\prime}\subset X-\{x_{n_{j}}:
j< i\}$ for $i\in \mathbb{N}$. Then $x_{n_{i}}$ belongs to $i$ many
elements of $(\mathcal{B})_{x}$, which is a contradiction.
\end{proof}

\vskip 0.5cm
\section{Inverse image of spaces with uniform bases at non-isolated points}
In this section, we mainly discuss the inverse image of spaces with
uniform bases at non-isolated points.

\begin{definition}
Let $X$ be a topological space.
\begin{enumerate}
\item $X$ is called a {\it $w\triangle$-space at non-isolated
points} if there exists a sequence $\{\mathcal{U}_{n}\}_{n}$ of open
covers such that, for every $x\in X-I$, whenever $x_{n}\in
\mbox{st}(x, \mathcal{U}_{n})$, then $\{x_{n}\}_{n}$ has a cluster
point.

\item $X$ is said to have a $G_{\delta}$-{\it diagonal at non-isolated
points} if there exists a sequence $\{\mathcal{U}_{n}\}_{n}$ of open
covers such that $\bigcap_{n\in\mathbb{N}}\mbox{st}(x,
\mathcal{U}_{n})=\{x\}$ for every $x\in X-I$. Moreover, $X$ is said
to have a $G_{\delta}^{\ast}$-{\it diagonal at non-isolated points}
if we replace ``$\bigcap_{n\in\mathbb{N}}\mbox{st}(x,
\mathcal{U}_{n})=\{x\}$'' by
``$\bigcap_{n\in\mathbb{N}}\overline{\mbox{st}(x,
\mathcal{U}_{n})}=\{x\}$''.
\end{enumerate}
\end{definition}

It is obvious that
\begin{enumerate}
\item $X$ is developable at non-isolated points~$\Rightarrow$ $X$ is a $\mbox{w}\triangle$-space at non-isolated points;

\item $X$ has a $G_{\delta}^{\ast}$-diagonal at non-isolated points~$\Rightarrow
X$ has a $G_{\delta}$-diagonal at non-isolated points;
\end{enumerate}

\begin{example}
There exists a perfect map from a space $X$ onto a metric space,
where $X$ has not any uniform base at non-isolated points.
\end{example}

\begin{proof}
Let $X=[0, 1]\times \{0, 1\}$ and endow $X$ with the lexicographic
ordered space. Let $f:X\rightarrow [0, 1]$ be a naturally projective
map, where $[0, 1]$ endowed with the usual topology. Since $X$ is
compact, $f$ is a closed and 2-to-one map. $X$ does not have an
uniform base at non-isolated points since $X$ has no uniform base
and does not contain any isolated points.

From this example it can be seen that a closed and 2-to-one map does
not inversely preserve spaces with an uniform base at non-isolated
points.
\end{proof}

\begin{example}
There exists an open and $\leq$2-to-one map from a space $X$ onto a
metric space, where $X$ has not any uniform base at non-isolated
points.
\end{example}

\begin{proof}
Y. Tanaka in \cite[Example 3.7]{Ta} constructed a regular space $X$
which is the inverse image of a compact metric space under an open
and $\leq$2-to-one map, but $X$ is not a first countable space.
Hence $X$ has not any uniform base at non-isolated points.
\end{proof}

\begin{example}
Open and closed map doesn't inversely preserve spaces with uniform
base at non-isolated points.
\end{example}

\begin{proof}
Let $X=[0, \omega_{1}]$ be an usually ordered space. Put
$f:X\rightarrow X/ X$ be a quotient map by identifying $X$ to a
single point. Then it is obvious that $f$ is an open and closed map.
But $X$ has not any uniform base at non-isolated points.
\end{proof}

We don't know whether spaces with an uniform base at non-isolated
points are inversely preserved by an open, closed and finite-to-one
map. So we have the following question.

\begin{question}
Are spaces with an uniform base at non-isolated points inversely
preserved by open, closed and finite-to-one maps?
\end{question}

By slightly modifying the proof in \cite[Theorem 6]{MH}, we can
obtain the following.

\begin{theorem}
Let $f:X\rightarrow Y$ be a closed, finite-to-one and local
homeomorphism map, where $Y$ has an uniform base at non-isolated
points. Then $X$ has an uniform base at non-isolated points.
\end{theorem}

It is well known that every open and $k$-to-one map is a closed and
locally homeomorphism map. Hence, we have the following corollary.

\begin{corollary}
Open and $k$-to-one maps inversely preserve spaces with an uniform
base at non-isolated points.
\end{corollary}

Finally, we consider the inverse image of spaces with an uniform
base at non-isolated points under the irreducible perfect maps.

\begin{lemma}\label{l5}
Let $X$ be regular and metacompact at non-isolated points. If
$\{\mathcal{U}_{n}\}_{n}$ is a sequence of open coverings of $X$,
then there exists a sequence $\{\mathcal{V}_{n}\}_{n}$ of open
coverings of $X$ such that, for any $y\in X^{d}$,
$\bigcap_{n\in\mathbb{N}}\overline{\mbox{st}(y,
\mathcal{V}_{n})}=\bigcap_{n\in\mathbb{N}}\mbox{st}(y,
\mathcal{V}_{n})\subset\bigcap_{n\in\mathbb{N}}\mbox{st}(y,
\mathcal{U}_{n})$.
\end{lemma}

\begin{proof}
Since $X$ is regular and metacompact at non-isolated points, there
exists a sequence $\{\mathcal{V}_{n}\}_{n}$ of open coverings of $X$
satisfying the following conditions:

(i)\, For each $n\in \mathbb{N}$, $\mathcal{V}_{n}$ is point-finite
at non-isolated points and refines
$$(\wedge_{i<n}\mathcal{V}_{i})\bigwedge (\wedge_{i\leq
n}(\mathcal{U}_{i});$$

(ii)\, For any $V\in \mathcal{V}_{n}$ and $i<n$, there exists a
$W\in \mathcal{V}_{i}$ such that $\overline{V}\subset W$.

Let $y\in X^{d}$. For each $n\in \mathbb{N}$, there are only
finitely many members of $\mathcal{V}_{n}$ which contains $y$. Hence
$\overline{\mbox{st}(y, \mathcal{V}_{n+1})}=\cup\{\overline{V}:y\in
V\in \mathcal{V}_{n+1}\}\subset \mbox{st}(y, \mathcal{V}_{n})$. Thus
$\bigcap_{n\in\mathbb{N}}\overline{\mbox{st}(y,
\mathcal{V}_{n})}=\bigcap_{n\in\mathbb{N}}\mbox{st}(y,
\mathcal{V}_{n})\subset\bigcap_{n\in\mathbb{N}}\mbox{st}(y,
\mathcal{U}_{n})$.
\end{proof}

\begin{lemma}\label{l6}
Let $X$ be a regular space, where $X$ has a $G_{\delta}$-diagonal at
non-isolated points. If $X$ is metacompact at non-isolated points,
then $X$ has a $G_{\delta}^{\ast}$-diagonal at non-isolated points.
\end{lemma}

\begin{proof}
It is easy to see by Lemma 3.8.
\end{proof}

\begin{lemma}\label{l7}
Let $X$ be a regular space, where $X$ has a
$G_{\delta}^{\ast}$-diagonal at non-isolated points. If $X$ is a
$\mbox{w}\triangle$-space at non-isolated points, then $X$ is a
developable space at non-isolated points.
\end{lemma}

\begin{proof}
let $\{\mathcal{U}_{n}\}_{n}$ and $\{\mathcal{V}_{n}\}_{n}$ be a
$G_{\delta}^{\ast}$-diagonal at non-isolated points and a
$\mbox{w}\triangle$-sequence at non-isolated points, respectively.
Then $\{\mathcal{U}_{n}\wedge\mathcal{V}_{n}\}_{n}$ is a development
at non-isolated points for $X$. Indeed, for any $x\in X-I$ and $x\in
U$ with $U\in\tau (X)$, there exists an $m\in\mathbb{N}$ such that
$x\in\mbox{st}(x, \mathcal{U}_{n}\wedge\mathcal{V}_{n})\subset U$.
Suppose not, then $\mbox{st}(x,
\mathcal{U}_{n}\wedge\mathcal{V}_{n})\not\subset U$ for any
$n\in\mathbb{N}$. We can choose a point $x_{n}\in\mbox{st}(x,
\mathcal{U}_{n}\wedge\mathcal{V}_{n})\setminus U $ for any
$n\in\mathbb{N}$. Since $\mbox{st}(x,
\mathcal{U}_{n}\wedge\mathcal{V}_{n})\subset\mbox{st}(x,
\mathcal{V}_{n})$, $x_{n}\in\mbox{st}(x, \mathcal{V}_{n})$. Hence
$\{x_{n}\}$ has a cluster point. Let $y$ be a cluster point of
$\{x_{n}\}$. Since $\mbox{st}(x,
\mathcal{V}_{n})\subset\overline{\mbox{st}(x, \mathcal{V}_{n})}$,
$y\in\overline{\mbox{st}(x, \mathcal{V}_{n})}$. Hence $y=x$ because
$\bigcap_{n\in\mathbb{N}}\overline{\mbox{st}(x,
\mathcal{V}_{n})}=\{x\}$. Thus $\{x_{n}\}$ has only one cluster
point $x$. But $x_{n}\notin U$ for any $n\in \mathbb{N}$, a
contradiction.
\end{proof}

\begin{lemma}\label{l7}
Let $f:X\rightarrow Y$ be an irreducible perfect map, where $Y$ is a
$\mbox{w}\triangle$-space at non-isolated points. Then $X$ is a
$\mbox{w}\triangle$-space at non-isolated points.
\end{lemma}

\begin{proof}
Let $\{\mathcal{U}_{n}\}_{n}$ be a $\mbox{w}\triangle$-sequence at
non-isolated points for $Y$. We only prove that
$\{f^{-1}(\mathcal{U}_{n})\}_{n}$ is a $\mbox{w}\triangle$-sequence
at non-isolated points for $X$. Let $x\in X-I(X)$ and $x_{n}\in
\mbox{st}(x, f^{-1}(\mathcal{U}_{n}))$ for each $n\in\mathbb{N}$.
Then $f(x_{n})\in \mbox{st}(f(x), \mathcal{U}_{n})$. Since $f$ is an
irreducible map, $f(x)\in Y-I(Y)$. Hence $\{f(x_{n})\}$ has a
cluster point in $Y$. Since $f$ is a perfect map, $\{x_{n}\}$ has a
cluster point in $X$. Hence $\{f^{-1}(\mathcal{U}_{n})\}_{n}$ is a
$\mbox{w}\triangle$-sequence at non-isolated points for $X$.
\end{proof}

\begin{lemma}\label{l8}
Let $f:X\rightarrow Y$ be an irreducible perfect map, where $X$ is
regular and has a $G_{\delta}$-diagonal. If $Y$ is metacompact at
non-isolated points, so is $X$.
\end{lemma}

\begin{proof}
Let $\mathcal{U}$ be an open covering for $X$. There exists
$\mathcal{U}(y)\in \mathcal{U}^{<\omega}$ such that
$f^{-1}(y)\subset \cup\mathcal{U}(y)$ for any $y\in Y$. Then there
exists an open neighborhood $V_{y}$ of $y$ such that
$f^{-1}(V_{y})\subset \cup\mathcal{U}(y)$. Since $\{V_{y}:y\in Y\}$
is an open covering for $Y$, there exists a point-finite open
refinement $\{W_{y}:y\in Y\}$ at non-isolated points such that
$W_{y}\subset V_{y}$ for any $y\in Y$. Hence $\{f^{-1}(W_{y})\cap
U:y\in Y, U\in \mathcal{U}(y)\}$ is a point-finite open refinement
at non-isolated points of $\mathcal{U}$.
\end{proof}

\begin{theorem}\label{t0}
Let $f:X\rightarrow Y$ be an irreducible perfect map, where $X$ is
regular and has a $G_{\delta}$-diagonal. If $Y$ has an uniform base
at non-isolated points, so does $X$.
\end{theorem}

\begin{proof}
It is easy to see by Lemmas~\ref{l5}, \ref{l6}, \ref{l7}, \ref{l8}
and~\ref{l4}.
\end{proof}

We don't know whether we can omit the condition ``irreducible map''
in Theorem~\ref{t0}. So we have the following question.

\begin{question}
Let $f:X\rightarrow Y$ be a perfect map, where $X$ is regular and
has a $G_{\delta}$-diagonal. If $Y$ has an uniform base at
non-isolated points, does $X$ have an uniform base at non-isolated
points?
\end{question}

{\bf Acknowledgements}. The authors would like to thank the referee
for his$\diagup$her valuable comments and suggestions.

\vskip0.9cm

\end{document}